\documentclass[reqno,11pt]{amsart}
\usepackage{amssymb,amsmath}
\usepackage{tikz,enumerate,mathrsfs,url}
\usepackage{color}

\usepackage{hyperref}

\numberwithin{equation}{section}
\newtheorem{theorem}[equation]{Theorem}
\newtheorem{lemma}[equation]{Lemma}

\newtheorem{corollary}[equation]{Corollary}

\theoremstyle{definition}

\newtheorem{example}[equation]{Example}
\newtheorem{identity}{Identity}

\theoremstyle{remark}
\newtheorem{remark}{Remark}

\newcommand{\Y}{\mathscr{Y}}
\newcommand{\seqnum}[1]{\href{http://oeis.org/#1}{\underline{#1}}}

\title[On the enumeration of restricted words]%
{On the enumeration of restricted words over \\ a finite alphabet}
\author{Daniel Birmajer}
\address{Department of Mathematics\\ Nazareth College\\ 4245 East Ave.\\ Rochester, NY 14618}
\author{Juan B. Gil}
\address{Penn State Altoona\\ 3000 Ivyside Park\\ Altoona, PA 16601}
\author{Michael D. Weiner}

\begin{document}
\maketitle

\begin{abstract}
We present a method for the enumeration of restricted words over a
finite alphabet. Restrictions are described through the inclusion or
exclusion of suitable building blocks used to construct the words by
concatenation. Our approach, which relies on the invert transform and
its representation in terms of partial Bell polynomials, allows us to
generalize and address in a systematic manner previous results in the
subject.
\end{abstract}

\section{Introduction}
\label{sec:introduction}

Motivated by a recent paper of Janji\'c \cite{Janjic15}, we examine the problem of enumerating words over a finite alphabet using the invert transform, partial Bell polynomials, and building blocks 
that allow us to effectively manage a wide variety of restrictions.

Recall that, given a sequence $x=(x_n)$, its {\sc invert} transform $\Y(x)=(y_n)$ is defined by
\begin{equation*}
 1+\sum_{n=1}^\infty y_n t^n = \frac{1}{1-\sum\limits_{n=1}^\infty x_n t^n},
\end{equation*} 
see, e.g., the paper by Cameron \cite{Cameron}, and the work by Bernstein and Sloane \cite{BernsteinSloane}. In short, if $X(t) = \sum_{n=1}^\infty x_n t^n$ and $Y(t) = \sum_{n=1}^\infty y_n t^n$, then 
\begin{equation*}
  1 + Y(t) = \frac{1}{1-X(t)}, \text{ and so }\, Y(t) = \frac{X(t)}{1-X(t)}.
\end{equation*} 
This map is clearly invertible on the space of linear recurrence sequences with constant coefficients, and we have $\Y^m(x) = \frac1m \Y(mx)$ for every $m\in\mathbb{N}$. Indeed, inductively, if $Y_{m}(t)$ denotes the generating function of $\Y^m(x)$, then
\begin{equation*}
  Y_{m}(t) = \frac{Y_{m-1}(t)}{1-Y_{m-1}(t)} = \frac{\frac{X(t)}{1-(m-1)X(t)}}{1-\frac{X(t)}{1-(m-1)X(t)}}= \frac{X(t)}{1-mX(t)}.
\end{equation*}

\medskip
The following statement is straightforward.

\begin{lemma}
For $b\in\mathbb{N}$, the invert transform of $x_n=b^{n-1}$ is given by $y_n=(b+1)^{n-1}$. Thus for $n>1$, $y_n$ counts the number of words of length $n-1$ over the alphabet $\{0,1,\dots,b\}$. Moreover, by applying the invert transform $m$ times, we get that $\Y^m(x)$ counts the words over the alphabet $\{0,1,\dots,b+m-1\}$.
\end{lemma}

\begin{corollary}
Let $x=(x_n)$ be a sequence of nonnegative integers. If $x_n\le b^{n-1}$ for all $n\in\mathbb{N}$, then the elements of $\Y^m(x)$ count some restricted words over the alphabet $\{0,1,\dots,b+m-1\}$.
\end{corollary}
\noindent
This is clear since, if $x_n\le w_n$ for every $n$, then $\big(1-\sum\limits_{n=1}^\infty x_n t^n\big)^{-1} \le \big(1-\sum\limits_{n=1}^\infty w_n t^n\big)^{-1}$.

\medskip
For $b=1$, several families of restricted words were considered by Janji\'c in \cite[Prop.~20]{Janjic15}. Consistent with op.\ cit., we let $f_m$ denote the $m$-th invert transform of a given sequence $f_0$, i.e., $f_m = \Y^m(f_0)$. In \cite[Corollary~24, 28, and 33]{Janjic15}, Janji\'c discussed three examples in which $f_0$ is a binary sequence. He posted the problem of finding, for an arbitrary sequence $f_0$, either a recurrence or an explicit formula for $f_m(n)$ in terms of $m$ and $n$, together with a description of the set of restricted words over the alphabet $\{0,1,\dots,m\}$ counted by $f_m(n)$. He also discussed two instances in which the sequence  $f_0$ is not binary but satisfies a second order linear recurrence relation. For all these examples, Janji\'c gives recurrence relations for $f_m$ obtained from the recurrence relation satisfied by $f_0$.

In this paper, we present a somewhat different method for the enumeration of restricted words that may be applied to a wide class of restrictions, see Section~\ref{sec:words}. Our approach relies on a family of building blocks that we utilize to construct words by concatenation, and on a representation of the invert transform in terms of partial Bell polynomials. We discuss several binary  and nonbinary examples (Sections~\ref{sec:binary} and \ref{sec:nonbinary}, respectively) and shed some light on the examples considered by Janji\'c \cite{Janjic15}. For all of the examples discussed here, we give closed combinatorial expressions. 

We finish the introduction by recalling that, by means of Fa{\`a} di Bruno's formula, the invert transform can be written in terms of partial Bell polynomials, see e.g.\ \cite{BGW15a}. Indeed, if $(y_n)$ is the invert transform of $x=(x_n)$, then
\begin{equation}\label{eqn:invertBell}
    y_n = \sum_{k=1}^n \frac{k!}{n!} B_{n, k} (1!x_1, 2! x_2, \dotsc) \;\text{ for } n\ge 1,
\end{equation}
where $B_{n,k}$ denotes the $(n,k)$-th partial Bell polynomial defined as
\begin{equation*}
  B_{n,k}(z_1,\dots,z_{n-k+1})=\sum_{\alpha\in\pi(n,k)} \frac{n!}{\alpha_1! \alpha_2!\cdots \alpha_{n-k+1}!}\left(\frac{z_1}{1!}\right)^{\alpha_1}\cdots \left(\frac{z_{n-k+1}}{(n-k+1)!}\right)^{\alpha_{n-k+1}}
\end{equation*}
with $\pi(n,k)$ denoting the set of multi-indices $\alpha\in{\mathbb N}_0^{n-k+1}$ such that $\alpha_1+\alpha_2+\cdots=k$ and $\alpha_1+2\alpha_2+3\alpha_3+\cdots=n$. For the basic properties of these polynomials, we refer to Comtet's book \cite[Section~3.3]{Comtet}. A large collection of Bell identities and convolution formulas is available throughout the literature, see e.g.\ \cite{BBK08, BGW12, WW09}.

From its definition, it is clear that $B_{n,k}$ is homogeneous of degree $k$:
\[  B_{n,k}(az_1,\dots,az_{n-k+1}) = a^k B_{n,k}(z_1,\dots,z_{n-k+1}). \]
Thus, since $\Y^m(x) = \frac1m \Y(mx)$ for every $m\in\mathbb{N}$, we get from \eqref{eqn:invertBell} that the terms $y_{m,n}$ of the $m$-th invert transform of $x=(x_n)$ can be written as
\begin{equation*}
    y_{m,n} = \sum_{k=1}^n \frac{k!}{n!} m^{k-1} B_{n, k} (1!x_1, 2! x_2, \dotsc) \;\text{ for } n\ge 1. 
\end{equation*}

As we will illustrate in the next sections, this representation of $\Y^m$ provides an effective tool to enumerate restricted words and leads to interesting combinatorial formulas.

\section{Restricted words}
\label{sec:words}

For a fixed $b\in\mathbb{N}$ we consider the following family of {\em building blocks}:
\begin{quote}
Let $W_b(j)$ be the set of words of length $j$, starting with $b$, followed by $j-1$ letters from the alphabet $\{0,1,\dots,b-1\}$.
\end{quote}
Observe that $|W_b(j)| = b^{j-1}$. Moreover, any word of length $n$ over the alphabet $\{0,1,\dots,b\}$, starting with $b$, can be made by unique concatenation of blocks in $W_b(j)$ for $j=1,\dots,n$.

\medskip
Given $b\in\mathbb{N}$ and a sequence $f_0$ of nonnegative integers, let $\{U_j\}$ be a fixed sequence of subsets $U_j\subseteq W_b(j)$ such that $|U_j|=f_0(j)$ for $j=1,2,\dots$, and let
\begin{quote}
$\mathscr{U}_{b,n}(f_0)$ be the set of words of length $n$ over the alphabet $\{0,1,\dots,b\}$, starting with $b$, made by concatenation of blocks from $U_j$ for $j=1,\dots,n$. If $f_0(j)=0$, no block from $U_j$ is used.
\end{quote}

\begin{theorem} \label{thm:basicwords}
Let $f_0$ be a sequence of nonnegative integers such that $f_0(j)\le b^{j-1}$ for every $j\in\mathbb{N}$ and some $b\in\mathbb{N}$. Then the sequence
\begin{equation*}
 f_1(0)=1,\;\; f_1(n) = \sum_{k=1}^{n} \frac{k!}{n!} B_{n, k} (1!f_0(1), 2! f_0(2), \dotsc),
\end{equation*} 
counts the number of words in $\mathscr{U}_{b,n}(f_0)$. Moreover, each term $\frac{k!}{n!} B_{n, k} (1!f_0(1), 2! f_0(2), \dotsc)$ counts the number of such words made with exactly $k$ blocks.
\end{theorem}

\begin{proof}
By definition, every word in $\mathscr{U}_{b,n}(f_0)$ can be written as a unique concatenation of blocks from a subset of $W_b(j)$ with $f_0(j)$ elements for $j=1,\dots,n$. Any of these words must start with a block from $W_b(j)$ for some $j$, followed by a word in $\mathscr{U}_{b,n-j}(f_0)$. Thus we have the  recurrence relation
\begin{equation*}
  f_1(n) = \sum_{j=1}^n f_0(j) f_1(n-j) \; \text{ with }\; f_1(0) = 1,
\end{equation*}
which at the level of generating functions implies
\begin{equation*}
 1+\sum_{n=1}^\infty f_1(n) t^n = \frac{1}{1-\sum\limits_{n=1}^\infty f_0(n) t^n}.
\end{equation*} 
In other words, the sequence $f_1$ is the invert transform of $f_0$, which as discussed in the introduction, can be written in terms of partial Bell polynomials, as claimed. Finally, the statement about the number of words made with exactly $k$ blocks is a direct consequence of the definition of the partial Bell polynomials.
\end{proof}

The case of binary building blocks ($b=1$) will be discussed with more details in the next section, and the importance of greater values of $b$ will be illustrated in Section~\ref{sec:nonbinary}. 

\begin{remark}
We will choose $b$ based on the type of restricted words we are interested in counting, not on the size of the alphabet. Once $b$ is fixed, we can enlarge the alphabet by repeatedly applying the invert transform.
\end{remark}

\medskip
Given $b\in\mathbb{N}$ and a sequence $f_0$ of nonnegative integers, let $\{U_j\}$ be a fixed sequence of subsets $U_j\subseteq W_b(j)$ such that $|U_j|=f_0(j)$ for $j=1,2,\dots$. Every word in $U_j$ has length $j$ and starts with $b$ followed by $j-1$ letters from the alphabet $\{0,1,\dots,b-1\}$. If we allow the character $b$ to be replaced by any element of the set $\{b,\dots,b+m-1\}$ for $m\in\mathbb{N}$, then every word in $U_j$ generates $m$ words over the alphabet $\{0,1,\dots,b+m-1\}$. Thus each $U_j$ gives rise to a set of words of length $j$, call it $U_j^m$, such that $|U_j^m| = m\cdot |U_j|$. We let

\begin{quote}
$\mathscr{U}_{b,n}^m(f_0)$ be the set of words of length $n$ over the alphabet $\{0,1,\dots,b+m-1\}$, \emph{starting with} $b$, made by concatenation of blocks from $U_j^m$ for $j=1,\dots,n$. 
If $f_0(j)=0$, no block from $U_j^m$ is used.
\end{quote}

\begin{corollary} \label{cor:m_words}
Let $f_0$ be a sequence of nonnegative integers such that $f_0(j)\le b^{j-1}$ for every $j\in\mathbb{N}$ and some $b\in\mathbb{N}$. For $m\in\mathbb{N}$, the sequence
\begin{equation*}
 f_m(0)=1,\;\; f_m(n) = \sum_{k=1}^{n} \frac{k!}{n!} m^{k-1} B_{n, k} (1!f_0(1), 2! f_0(2), \dotsc),
\end{equation*} 
counts the number of words in $\mathscr{U}_{b,n}^m(f_0)$. Note that $f_m$ is the $m$-th invert transform of $f_0$. 
\end{corollary}

\begin{proof}
By Theorem~\ref{thm:basicwords}, the expression $\frac{k!}{n!} B_{n, k} (1!f_0(1), 2! f_0(2), \dotsc)$ counts the number of words in $\mathscr{U}_{b,n}(f_0)$ of the form $w_{j_1}\cdots w_{j_k}$, where each $w_{j_i}$ is a subword in $U_{j_i}$ for $i=1,\dots,k$. If each of the subwords $w_{j_2},\dots,w_{j_k}$ is allowed to be replaced by any of the corresponding $m$ blocks in $U_{j_2}^m,\dots,U_{j_k}^m$, respectively, we obtain $\frac{k!}{n!} m^{k-1}B_{n, k} (1!f_0(1), 2! f_0(2), \dotsc)$ words in $\mathscr{U}_{b,n}^m(f_0)$. Note that every word in $\mathscr{U}_{b,n}^m(f_0)$ can be constructed in this manner.

Adding over $k$ from $1$ to $n$, we obtain the total number of elements of $\mathscr{U}_{b,n}^m(f_0)$.
\end{proof}

\section{Restricted words of binary type}
\label{sec:binary}

For $b=1$, the building blocks in $W_b(j)$ are words of the form $10\cdots 0$ with $j-1$ zeros. Then, for $m\in\mathbb{N}$ and a given binary sequence $f_0$, the set 
\begin{quote}
$\mathscr{U}_{1,n}^m(f_0)$ consists of words of length $n$ over the alphabet $\{0,1,\dots,m\}$, starting with 1, made by concatenating blocks of the form $k\,0\cdots 0$ with $j-1$ zeros, $k\in\{1,\dots,m\}$, where blocks of length $j$ are used only if $f_0(j)=1$.
\end{quote}
For example, if $f_0$ is the sequence $1,0,1,1,0,0,\dots$, then the elements of $\mathscr{U}_{1,5}^2(f_0)$ are words of length 5 over the alphabet $\{0,1,2\}$, starting with 1, made by using only the building blocks of length 1, 3, and 4:
\begin{center}
 $1, 2 \in U_1^2$, \,$100, 200 \in U_3^2$, and $1000, 2000\in U_4^2$.
\end{center}
Moreover, the set $\mathscr{U}_{1,5}^2(f_0)$ consists of 32 such words:

\medskip
\begin{center}
\parbox{.65\textwidth}{
11111, 11112, 11121, 11211, 12111, \\
11122, 11212, 11221, 12112, 12121, 12211, \\
11222, 12122, 12212, 12221, 12222, \\
10011, 10012, 10021, 10022, 11001, 11002, 12001, 12002, \\
11100, 11200, 12100, 12200, 10001, 10002, 11000, 12000.
}
\end{center}
\medskip

Observe that the block restrictions given through the sequence $f_0$ carry over to the words in $\mathscr{U}_{1,n}^m(f_0)$. For instance, in the above example, the sequence $f_0$ induces the restrictions that $\mathscr{U}_{1,5}^2(f_0)$ contains no words with an isolated zero ($f_0(2)=0$) and no words with more than three consecutive zeros ($f_0(j)=0$ for $j\ge 5$).

\begin{remark}
By treating the first 1 as a dummy character, we get a natural bijection between $\mathscr{U}_{1,n+1}^m(f_0)$ and the set of all words of length $n$ with the restrictions determined by $f_0$.
\end{remark}

We proceed to illustrate the results from Section~\ref{sec:words} with a few examples. In the remaining part of this section, we will freely use some basic properties of partial Bell polynomials found in Comtet's book \cite[Section~3.3]{Comtet}. More intricate identities will be proved in the Appendix.

\begin{example}[{cf.\ \cite[Cor.~33]{Janjic15}}]
Suppose we wish to count the number of words of length $n$ over the alphabet $\{0, 1, \dots, m\}$ avoiding runs of $\ell$ zeros, $\ell\ge 1$. This can be done by computing $f_m(n+1)=\big|\mathscr{U}_{1,n+1}^m(f_0)\big|$ with
\begin{equation*}
  f_0(j) = 
 \begin{cases}
 1, & \text{if}\; 1\le j \le \ell; \\
 0, &\text{otherwise}.
 \end{cases}
\end{equation*}
By Corollary~\ref{cor:m_words}, we have
\begin{align*}
 f_m(n+1) &=  \sum_{k=1}^{n+1} \frac{k!}{(n+1)!} m^{k-1}\, B_{n+1, k} (1!, 2!, \dotsc , \ell!, 0,  \dotsc).
\end{align*}
By Identity~\ref{Bell_Identity1} in the Appendix, we have
\begin{equation*}
 B_{n+1, k} (1!, 2!, \dotsc , \ell!, 0,  \dotsc) = \frac{(n+1)!}{k!}
 \sum_{j=0}^{\lfloor\frac{n+1-k}{\ell}\rfloor} (-1)^j \binom{k}{j} \binom{n-\ell j}{k-1},
\end{equation*}
and therefore,
\begin{align*}
 f_m(n+1) &=  \sum_{k=1}^{n+1}
 \sum_{j=0}^{\lfloor\frac{n+1-k}{\ell}\rfloor} (-1)^j \binom{k}{j} \binom{n-\ell j}{k-1} m^{k-1}.
\end{align*}

For various values of $\ell$ and  $m$, this gives the sequences 
\seqnum{A000045}, \seqnum{A028859},  \seqnum{A125145}, \seqnum{A086347}, \seqnum{A180033},
\seqnum{A000073}, \seqnum{A119826}, \seqnum{A000078}, \seqnum{A209239}, \seqnum{A000322}
in Sloane \cite{Sloane}.
\end{example}

\begin{example} [{cf.\ \cite[Cor.~28]{Janjic15}}]
Consider words of length $n$ over the alphabet $\{0, 1, \dotsc, m\}$ avoiding runs of zeros of odd length. By adding a 1 to the left of each of these words, we obtain the set $\mathscr{U}_{1,n+1}^m(f_0)$ with
\begin{equation*}
  f_0(j) = 
 \begin{cases}
 1, &\text{if $j$ is odd};\\ 
 0, &\text{if $j$ is even}.
\end{cases}
\end{equation*}
In this case, $f_m(n+1)=\big|\mathscr{U}_{1,n+1}^m(f_0)\big|$ is given by
\begin{equation*}
  f_m(n+1) = \sum_{k=1}^{n+1} \frac{k!}{(n+1)!}m^{k-1}\, B_{n+1, k} (1!,0,3!,0,5!, \dotsc).
\end{equation*}
It is easy to see that $B_{n+1, k} (1!,0,3!,0,5!, \dotsc)=0$ unless $n+1+k$ is even. Moreover, 
\begin{align*}
 B_{n+1, k} (1!,0, 3!, 0, 5!, 0,\dotsc) 
 &= \frac{(n+1)!}{(n+1-k)!} B_{n+1-k,k}(0,2!,0,4!,0,\dots) \\
 &= \frac{(n+1)!}{(\frac{n+k+1}{2})!} B_{\frac{n+k+1}{2},k}(1!,2!,3!,\dots) = \frac{(n+1)!}{k!}\binom{\frac{n+k-1}{2}}{k-1}.
\end{align*}
Hence the number of elements in $\mathscr{U}_{1,n+1}^m(f_0)$ is given by
\begin{equation*}
  f_m(n+1) = \sum_{\substack{1\le k\le n+1\\ n+k+1\,\equiv\, 0\; (\text{mod }2)}} \binom{\frac{n+k-1}{2}}{k-1}m^{k-1} = \sum_{k=0}^{\lfloor \frac{n}{2}\rfloor} \binom{n-k}{k}m^{n-2k}. 
\end{equation*}

For $m=1,\dots,10$, this gives the sequences
\seqnum{A000045}, \seqnum{A000129}, \seqnum{A006190}, \seqnum{A001076}, \seqnum{A052918},
\seqnum{A005668}, \seqnum{A054413}, \seqnum{A041025}, \seqnum{A099371}, \seqnum{A041041} 
in Sloane \cite{Sloane}.
\end{example}

\begin{example}
For $\ell \in\mathbb{N}$ let 
\begin{equation*}
  f_0(j) = 
 \begin{cases}
 0, &\text{if}\; j \le \ell;\\ 
 1, &\text{otherwise}.
\end{cases}
\end{equation*}
Then the sequence
\begin{equation*}
  f_m(n) = \sum_{k=1}^{n} \frac{k!}{n!}m^{k-1}\, B_{n, k} (0, \dotsc,0, (\ell +1)!, (\ell+2)!, \dotsc) 
\end{equation*}
counts the number of words of length $n$ over the alphabet $\{0, 1, \dotsc, m\}$, starting with 1, such that every nonzero element is followed by at least $\ell$ zeros. 

Since
\[ B_{n,k}(0,\dots,0, (\ell+1)!,(\ell+2)!,\dots) = \frac{n!}{(n-\ell k)!} B_{n-\ell k,k}(1!,2!\dots) \text{ for } \ell k<n, \]
and since $B_{n-\ell k,k}(1!,2!,\dots)=\frac{(n-\ell k)!}{k!}\binom{n-\ell k-1}{k-1}$, we get
\begin{equation*}
  \frac{k!}{n!} B_{n,k}(0,\dots,0, (\ell+1)!,(\ell+2)!,\dots) = \binom{n-\ell k-1}{k-1} \;\text{ for } \ell k< n.
\end{equation*}
Note that $B_{n,k}(0,\dots,0, (\ell+1)!,(\ell+2)!,\dots)=0$ for $\ell k\ge n$. Thus we arrive at
\begin{equation*}
  f_m(n) = \sum_{k=1}^{\lfloor \frac{n-1}{\ell}\rfloor} \binom{n-\ell k-1}{k-1} m^{k-1}.
\end{equation*}

In this case, all words in $\mathscr{U}_{1,n}^m(f_0)$ for $n\ge 2\ell+1$ start with a 1 followed by at least $\ell$ zeros, and end with $\ell$ zeros. By removing these $2\ell+1$ common characters, we get words of length $n-2\ell-1$ such that every nonzero element not in the last position is followed by at least $\ell$ zeros. In conclusion, the number
\begin{equation*}
  f_m(n+2\ell+1) = \sum_{k=1}^{\lfloor \frac{n+2\ell}{\ell}\rfloor} \binom{n-\ell(k-2)}{k-1} m^{k-1}
\end{equation*}
counts the number of words of length $n$ over the alphabet $\{0, 1, \dotsc, m\}$ having at least $\ell$ zeros between any two nonzero elements.

For $\ell=1$ this example was discussed by Janji\'c \cite[Cor.~24]{Janjic15}, and for $\ell=2$ and $m=1$, the sequence $f_1(n+3)$ gives the Narayana cows sequence, \seqnum{A000930} in Sloane \cite{Sloane}. For $\ell=2$ and $m=2,\dots,4$, we get the sequences \seqnum{A003229}, \seqnum{A084386}, \seqnum{A089977}. 
Other values of $\ell$ and $m$ lead to: \seqnum{A003269}, \seqnum{A052942}, \seqnum{A143454},
\seqnum{A003520}, \seqnum{A143447}, \seqnum{A143455}, \seqnum{A005708}.
 \end{example}

\medskip
We finish the section with a couple of more elaborate examples.

\begin{example}
Let 
\begin{equation*}
  f_0(j) = 
 \begin{cases}
 1, &\text{if } j = 1 \text{ or } j=\ell+1; \\ 
 0, &\text{otherwise}.
\end{cases}
\end{equation*}
The sequence $f_m(n+1)=\big|\mathscr{U}_{1,n+1}^m(f_0)\big|$ counts the number of words of length $n$ over the alphabet $\{0, 1, \dotsc, m\}$ containing zeros only in substrings of length $\ell$. By Corollary~\ref{cor:m_words},
\begin{equation*}
  f_m(n+1) = \sum_{k=1}^{n+1} \frac{k!}{(n+1)!}m^{k-1}\, B_{n+1, k} (1!, 0, \dotsc, (\ell + 1)!, 0 \dotsc).
\end{equation*}
By Identity~\ref{Bell_Identity2} in the Appendix, we have
\begin{equation*}
 B_{n, k} (1!, 0, \dotsc, (\ell + 1)!, 0 \dotsc ) 
 = \binom{n}{n-k+\frac{n-k}{\ell}} \frac{(n-k+\frac{n-k}{\ell})!}{(\frac{n-k}{\ell})!} \;\text{ for } n - k \,\equiv\, 0\; (\text{mod }\ell),
\end{equation*}
and therefore,
\begin{equation*}
  f_m(n+1) = \sum_{\substack{1\le k\le n+1\\ n+1 \,\equiv\, k\; (\text{mod }\ell)}} \!\!\frac{k!}{(n+1)!}\, m^{k-1} 
  \binom{n+1}{n+1-k+\frac{n+1-k}{\ell}} \frac{(n+1-k+\frac{n+1-k}{\ell})!}{(\frac{n+1-k}{\ell})!}.
\end{equation*}
Writing $n+1 = q\ell  + r$ with $0\le r < \ell$, we  obtain
\begin{align*}
  f_m(n+1) &= \!\!\sum_{\substack{1\le k\le n+1\\ n+1 \,\equiv\, k\; (\text{mod }\ell)}} \!\frac{k!}{(q\ell+ r)!} m^{k-1} \binom{q\ell + r}{q\ell + r-k+\frac{q\ell + r-k}{\ell}} \frac{(q \ell + r- k+\frac{q\ell + r-k}{\ell})!}{(\frac{q \ell + r-k}{\ell})!} \\ 
&= \sum_{i=0}^q \frac{(\ell i + r)!}{(q\ell  + r)!}  m^{\ell i + r -1}
  \binom{q\ell + r}{(\ell +1) (q-i)} \frac{((\ell+1) (q- i))!}{(q-i)!}  
 = \sum_{i=0}^q \binom{\ell i+r}{q-i} m^{\ell i + r -1}.
\end{align*}
In conclusion, the number of words of length $n$ over the alphabet $\{0, 1, \dotsc, m\}$ containing zeros only in substrings of length $\ell$, is given by
\begin{equation*}
  f_m(n+1) = \sum_{i=0}^{\lfloor \frac{n+1}{\ell}\rfloor} \binom{n+1-\ell i}{i} m^{n - \ell i}.
\end{equation*}

For $\ell=2$ and $m=1,\dots,6$, this gives the sequences
\seqnum{A068921}, \seqnum{A255115}, \seqnum{A255116}, \seqnum{A255117}, \seqnum{A255118}, 
\seqnum{A255119} in Sloane \cite{Sloane}. And for $m=1$ and $\ell=3,\dots,7$, we get
\seqnum{A003269}, \seqnum{A003520}, \seqnum{A005708}, \seqnum{A005709}, \seqnum{A005710}.
\end{example}

\begin{example}
For $\ell \in\mathbb{N}$ let 
\begin{equation*}
  f_0(j) = 
 \begin{cases}
 0, &\text{if}\; j = \ell + 1;\\ 
 1, &\text{otherwise}.
\end{cases}
\end{equation*}
Then the sequence
\begin{equation*}
  f_m(n+1) = \sum_{k=1}^{n+1} \frac{k!}{(n+1)!}m^{k-1}\, B_{n+1, k} (1!,  \dotsc, \ell!, 0, (\ell +2)!, (\ell+3)!, \dotsc) 
\end{equation*}
counts the number of words of length $n$ over the alphabet $\{0, 1, \dotsc, m\}$ avoiding runs of exactly $\ell$ zeros. By means of Identity~\ref{Bell_Identity3} in the Appendix, the partial Bell polynomials can be written as
\begin{align*}
B_{n+1, k} (1!,  \dotsc, \ell!,\, & 0, (\ell +2)!, (\ell+3)!, \dotsc) \\
&= \frac{(n+1)!}{k!} \left(\sum_{\kappa< k} (-1)^\kappa \binom{k}{\kappa} 
 \binom{n-(\ell+1)\kappa}{k-\kappa-1}+ (-1)^k \delta_{n+1,(\ell+1)k} \right),
\end{align*}
where $\delta_{n+1,(\ell+1)k}$ is the Kronecker delta function. Thus
\begin{equation*}
 f_m(n+1) = \sum_{k=1}^{n+1} \left(\sum_{\kappa< k} (-1)^\kappa \binom{k}{\kappa} 
 \binom{n-(\ell+1)\kappa}{k-\kappa-1} + (-1)^k \delta_{n+1,(\ell+1)k} \right) m^{k-1}.
\end{equation*} 
If $n+1=(\ell+1)q$, then
\begin{equation*}
 f_m(n+1) = (-1)^{q}m^{q-1} + \sum_{\kappa=0}^{q-1}\;  \sum_{k=\kappa+1}^{n+1-\ell\kappa} 
 (-1)^\kappa \binom{k}{\kappa}\binom{n-(\ell+1)\kappa}{k-\kappa-1} m^{k-1},
\end{equation*} 
and if $n+1\not\equiv 0$ modulo $(\ell+1)$, then
\begin{equation*}
 f_m(n+1) = \sum_{\kappa=0}^{\lfloor \frac{n+1}{\ell+1}\rfloor}\; \sum_{k=\kappa+1}^{n+1-\ell\kappa} 
 (-1)^\kappa  \binom{k}{\kappa}\binom{n-(\ell+1)\kappa}{k-\kappa-1} m^{k-1}.
\end{equation*}

For various values of $\ell$ and  $m$, this gives the sequences
\seqnum{A005251}, \seqnum{A120925}, \seqnum{A255813}, \seqnum{A255814}, \seqnum{A255815}, 
\seqnum{A049856}, \seqnum{A108758} in Sloane \cite{Sloane}.
\end{example}

\section{Further (nonbinary) examples}
\label{sec:nonbinary}

In \cite[Corollary 37]{Janjic15}, Janji\'c discusses $01$-avoiding words over a finite alphabet by considering the $m$-th invert transform of the sequence $f_0(j)=j$ for $j\in\mathbb{N}$. He verifies that $f_m$ satisfies the same recurrence relation as the sequence counting such words, but there is no other indication of how this $f_0$ is combinatorially related to the given restriction. 

In this section, we will use our results from Section~\ref{sec:words} to shed some light on the counting of $01$-avoiding words and will discuss examples that lead to the invert transform of the triangle numbers and other more general figurate numbers.

\begin{example}[{cf.\ \cite[Cor.~37]{Janjic15}}] \label{01avoiding}
We start by considering words of length $n$ over the alphabet $\{0,1,2\}$, starting with 2, that avoid the pattern $01$. With the terminology from Section~\ref{sec:words}, this can be done with blocks in $W_b(j)$ with $b=2$ and $j=1,\dots,n$.  For example, for $j=1,\dots,4$, we have the building blocks

\medskip
\begin{center}
\begin{tabular}{c@{\hspace{2em}}c@{\hspace{2em}}c@{\hspace{2em}}c}
$W_2(1)$ & $W_2(2)$ & $W_2(3)$ & $W_2(4)$ \\ \hline
&&& \\[-10pt]
2 & 20 & 200 & 2000 \\
& 21 & \textcolor{red}{\em 201} & \textcolor{red}{\em 2001} \\
&& 210 & \textcolor{red}{\em 2010} \\
&& 211 & \textcolor{red}{\em 2011} \\
&&& 2100 \\
&&& \textcolor{red}{\em 2101} \\
&&& 2110 \\
&&& 2111
\end{tabular}
\end{center}
\medskip

\noindent
and the colored blocks are the ones we wish to avoid. Observe that, once these building blocks are removed, every word made by concatenation of the remaining blocks will be a word that starts with 2 and avoids the subword $01$. Also note that the building blocks avoiding $01$ are precisely those where the digits are in nonincreasing order.

Let $U_j$ be the subset of $W_2(j)$ consisting of the $j$ building blocks whose digits are in nonincreasing order. With this choice, the number of $01$-avoiding words of length $n$, starting with 2, is equal to the number of elements in $\mathscr{U}_{2,n}(f_0)$ with $f_0$ given by $f_0(j)=j$ for $j=1,2,\dots$, recall the definition before Theorem~\ref{thm:basicwords}. By removing the starting digit 2 from each word in $\mathscr{U}_{2,n}(f_0)$, we get a natural bijection to the set of $01$-avoiding words of length $n-1$ over the alphabet $\{0,1,2\}$. 

The number of such words is then counted by the invert transform of $f_0$, see Theorem~\ref{thm:basicwords}. Moreover, the set $\mathscr{U}_{2,n}^m(f_0)$ for $m\ge 1$ contains all $01$-avoiding words of length $n$, starting with 2, over the larger alphabet $\{0,1,\dots,m+1\}$. If $f_m(n)=|\mathscr{U}_{2,n}^m(f_0)|$, Corollary~\ref{cor:m_words} together with Identity~\ref{Bell_Identity4} with $r=0$ (see Appendix) give
\begin{equation*}
 f_m(n) = \sum_{k=1}^{n} \frac{k!}{n!} m^{k-1} B_{n, k} (1!\cdot 1, 2!\cdot 2, 3!\cdot 3,\dotsc)
 =\sum_{k=1}^{n} \binom{n+k-1}{n-k} m^{k-1} \;\text{ for } n>1.
\end{equation*} 

In conclusion, if $f_0(j)=j$ for $j\in\mathbb{N}$, then for $n>1$ the number of $01$-avoiding words of length $n-1$ over the alphabet $\{0,1,\dots,m+1\}$ is given by
\begin{equation*}
 f_m(n) = \sum_{k=1}^{n} \binom{n+k-1}{n-k} m^{k-1}.
\end{equation*} 

For $m=1,\dots,10$, this gives the sequences
\seqnum{A001906}, \seqnum{A001353}, \seqnum{A004254}, \seqnum{A001109}, \seqnum{A004187}, \seqnum{A001090}, \seqnum{A018913}, \seqnum{A004189}, \seqnum{A004190}, \seqnum{A004191}
in Sloane \cite{Sloane}.
\end{example}

\begin{corollary}
It was noted in \cite[Corollary~37]{Janjic15} that the above sequence $f_m(n)$ coincides with $U_{n-1}(\tfrac{m+2}{2})$, where $U_{n}$ is the $n$-th Chebyshev polynomial of the second kind. As a consequence, we obtain the  identity:
\begin{equation*}
 U_{n}(\tfrac{x+2}{2}) = \sum_{k=0}^{n} \binom{n+k+1}{n-k} x^{k}.
\end{equation*} 
\end{corollary}

\begin{example} \label{010212avoiding}
Following the idea of the last example, we now let $b=3$ and let $U_j$ be the building blocks in $W_3(j)$ whose digits are in nonincreasing order. For instance, out of the 9 blocks in $W_3(3)$:
\begin{center} 
  300\quad \textcolor{red}{\em 301}\quad \textcolor{red}{\em 302}\quad 310\quad 311\quad 
  \textcolor{red}{\em 312}\quad 320\quad 321\quad 322 
\end{center}
there are 3 blocks that do not satisfy the restriction, so $U_3$ has only 6 elements. In general, for every $j\in\mathbb{N}$ there are $\binom{j+1}{2}$ such blocks, and they are precisely the building blocks that avoid the patterns $01$, $02$, and $12$. Thus with $f_0(j)=\binom{j+1}{2}$ and $U_j$ as above, the set $\mathscr{U}_{3,n}(f_0)$ consists of words of length $n$ over the alphabet $\{0,1,2,3\}$, starting with 3, and avoiding the patterns $01$, $02$, and $12$. Moreover, as discussed in Section~\ref{sec:words}, the set $\mathscr{U}_{3,n}^m(f_0)$ for $m\ge 1$ consists of words over the alphabet $\{0,1,\dots,m+2\}$, starting with 3, with the same restrictions as the words in $\mathscr{U}_{3,n}(f_0)$. Note that $\mathscr{U}_{3,n}^m(f_0)$ has the same number of elements as the set of words of length $n-1$ over $\{0,1,\dots,m+2\}$, avoiding $01$, $02$, and $12$.

Finally, by means of Corollary~\ref{cor:m_words} and Identity~\ref{Bell_Identity4} with $r=1$ (see Appendix), we conclude that for $n>1$ the sequence
\begin{equation*}
 f_m(n) = \sum_{k=1}^{n} \frac{k!}{n!} m^{k-1} B_{n, k} (1!\tbinom{2}{2}, 2!\tbinom{3}{2}, 3!\tbinom{4}{2},\dotsc)
 =\sum_{k=1}^{n} \binom{n+2k-1}{n-k} m^{k-1}
\end{equation*} 
gives the number of words of length $n-1$ over the alphabet $\{0,1,\dots,m+2\}$, avoiding the patterns $01$, $02$, and $12$.

For $m=1,2$, we get the sequences \seqnum{A052529} and \seqnum{A200676} in Sloane \cite{Sloane}.
\end{example}

\begin{remark}
In each of the previous two examples, $f_0$ is of the form $f_0(j)=\binom{j+r}{r+1}$ for $r\ge 0$. Example~\ref{01avoiding} corresponds to $r=0$ and Example~\ref{010212avoiding} is the case when $r=1$. Choosing $b=r+1$ for the construction of $\mathscr{U}_{b,n}^m(f_0)$, and choosing building blocks with a nonincreasing restriction on its digits, one gets words that avoid certain patterns.
\end{remark}

\begin{example}
Let $r\in\mathbb{N}$ and $f_0(j)=\binom{j+r}{r+1}$ for $j\ge 1$. Then for $n>1$, $f_m(n)=|\mathscr{U}_{r+1,n}^m(f_0)|$ counts the number of words of length $n-1$ over the alphabet $\{0,1,\dots,m+r+1\}$, avoiding all patterns of the form $a_1a_2$ with $a_1,a_2\in\{0,\dots,r+1\}$ and $a_1<a_2$. Applying Corollary~\ref{cor:m_words} together with Identity~\ref{Bell_Identity4} from the Appendix, we then get
\begin{equation*}
 f_m(n) =\sum_{k=1}^{n} \binom{n+(r+1)k-1}{n-k} m^{k-1}.
\end{equation*} 

For $m=1$ and $r=2,\dots,5$, this gives the sequences
\seqnum{A055991}, \seqnum{A079675}, \seqnum{A099242}, \seqnum{A099253} in Sloane \cite{Sloane}.
\end{example}

\bigskip
We finish this section by revisiting another example considered by Janji\'c \cite[Cor.~41]{Janjic15}. While our final formula does not seem to give information that could have not been obtained from the recurrence relation given in \cite{Janjic15}, the goal here is to provide more combinatorial insight and to illustrate the versatility of the building block approach introduced in Section~\ref{sec:words}.

\begin{example}[{cf.\ \cite[Cor.~41]{Janjic15}}]
For $q\in\mathbb{N}$ consider the set of $ii$-avoiding words of length $n$ over the alphabet $\{0,1,\dots,q\}$ for $i\in\{0,1,\dots,q-1\}$. This set of words can be completely described with building blocks in $W_q(j)$ for $j=1,\dots,n+1$. The elements of $W_q(j)$ are strings of length $j$, starting with $q$, followed by $j-1$ letters from the alphabet $\{0,1,\dots,q-1\}$. Clearly, there is only one word in $W_q(1)$, and there are $q$ words in $W_q(2)$,
 $q0, q1, \dots, q(q-1)$, none of which is of the form $ii$. Moreover, for $j>2$ there are exactly $q\cdot (q-1)^{j-2}$ $ii$-avoiding words in $W_q(j)$. Thus, if $f_0$ is defined by
\begin{equation*}
  f_0(j) =
 \begin{cases}
 1, & \text{if}\; j=1; \\
 q(q-1)^{j-2}, &\text{if}\; j\ge2;
 \end{cases}
\end{equation*}
then, by Theorem~\ref{thm:basicwords}, the sequence $f_1(n+1)=\sum_{k=1}^{n+1} \frac{k!}{(n+1)!} B_{n+1,k}(1!f_0(1), 2!f_0(2),\dots)$ gives the number of $ii$-avoiding words of length $n$ over $\{0,1,\dots,q\}$ for 
$0\le i<q$.

As before, Corollary~\ref{cor:m_words} gives a corresponding formula for the $m$-th invert transform, but before we work it out, let us simplify the partial Bell polynomial:
\begin{align*}
 B_{n+1, k} (1!f_0(1), 2! f_0(2), \dotsc)
 &=  B_{n+1, k} (1!, 2! q,3!q(q-1), 4!q(q-1)^2, \dotsc) \\
 &= \sum_{\ell=0}^k \frac{(n+1)!}{(n+1-k)!\ell!} B_{n+1-k, k-\ell} (q,2!q(q-1), 3!q(q-1)^2, \dotsc) \\
 &= \sum_{\ell=0}^k \frac{(n+1)!}{(n+1-k)!\ell!} q^{k-\ell}(q-1)^{n+1+\ell-2k}
 B_{n+1-k, k-\ell} (1!, 2!, \dotsc),
\end{align*} 
which for $k\le n$ implies 
\begin{equation*}
 \frac{k!}{(n+1)!} B_{n+1, k} (1!f_0(1), 2! f_0(2), \dotsc) 
 = \sum_{\ell=0}^{k-1} \binom{k}{\ell} \binom{n-k}{k-\ell-1} q^{k-\ell}(q-1)^{n+1+\ell-2k}.
\end{equation*}
Finally, by Corollary~\ref{cor:m_words}, we conclude that the number of $ii$-avoiding words of length $n$ over the alphabet $\{0,1,\dots,q+m-1\}$ for $i\in\{0,1,\dots,q-1\}$, is given by
\begin{equation*}
 f_m(n+1) = m^{n} + \sum_{k=1}^{n} 
 \sum_{\ell=0}^{k-1}\binom{k}{\ell} \binom{n-k}{k-\ell-1}q^{k-\ell}(q-1)^{n+1+\ell-2k} m^{k-1}.
\end{equation*} 

For various values of $m$ and $q$, this formula leads to the sequences
\seqnum{A000045}, \seqnum{A001333}, \seqnum{A055099}, \seqnum{A123347}, \seqnum{A180035} in Sloane \cite{Sloane}.
\end{example}

\section{Appendix: Bell polynomial identities}
\label{sec:appendix}

\begin{identity}\label{Bell_Identity1}
For $\ell\in\mathbb{N}$ we have
\begin{equation*}
 B_{n, k} (1!, 2!, \dotsc , \ell!, 0,  \dotsc) = \frac{n!}{k!} \sum_{j=0}^{\lfloor\frac{n-k}{\ell}\rfloor} (-1)^j \binom{k}{j} \binom{n-\ell j-1}{k-1}.
\end{equation*}
\end{identity}
A proof of this identity can be found in \cite[Theorem~4.1 combined with Theorem~2.1]{BBK08}.

\begin{identity}\label{Bell_Identity2}
For $\ell\in\mathbb{N}$ with $\ell>1$, we have
\begin{equation*}
 B_{n, k} (1!, 0, \dotsc, (\ell + 1)!, 0 \dotsc ) = \binom{n}{n-k+\frac{n-k}{\ell}} \frac{(n-k+\frac{n-k}{\ell})!}{(\frac{n-k}{\ell})!} 
\end{equation*}
whenever $n - k \,\equiv\, 0\; (\text{mod }\ell)$.
\end{identity}
\begin{proof}
Using \cite[{Eqs.\ (3n) and (3n'), Section~3.3}]{Comtet}, we have
\begin{align*}
B_{n, k} (1!, 0, \dotsc, (\ell + 1)!, 0 \dotsc ) 
&= \sum_{\kappa =0}^k\sum_{\nu = \kappa}^ n \binom{n}{\nu} B_{\nu, \kappa}(0 \dotsc, (\ell + 1)!, 0 \dotsc )  B_{n-\nu, k-\kappa}(1, 0,  0, \dotsc ) \\ 
&= \sum_{\kappa =0}^k \binom{n}{n-k+\kappa} B_{n-k+\kappa, \kappa}(0 \dotsc, (\ell+1)!, 0 \dotsc )  \\ 
&= \binom{n}{n-k+\frac{n-k}{\ell}} \frac{(n-k+\frac{n-k}{\ell})!}{(\frac{n-k}{\ell})!} \;\text{ for } n - k \,\equiv\, 0\; (\text{mod }\ell).
\end{align*}
\end{proof}

\begin{identity}\label{Bell_Identity3}
For $\ell\in\mathbb{N}$,
\begin{multline*}
\quad B_{n, k} (1!,  \dotsc, \ell!, 0, (\ell +2)!, (\ell+3)!, \dotsc) \\
= \frac{n!}{k!} \left(\sum_{\kappa=0}^{k-1} (-1)^\kappa \binom{k}{\kappa} 
 \binom{n-(\ell+1)\kappa-1}{k-\kappa-1}+ (-1)^k \delta_{n,(\ell+1)k} \right)\!, \qquad
\end{multline*}
where $\delta_{n,(\ell+1)k}$ is the Kronecker delta function.
\end{identity}
\begin{proof}
Let $\bar x_\ell$, $\bar e_\ell$, and $\bar x$ be defined as 
\begin{center}
$\bar x_\ell=(1!,\dots,\ell!,0,(\ell+2)!,\dots)$, \ $\bar e_\ell=(0,\dots,0,(\ell+1)!,0,\dots)$, and 
$\bar x = \bar e_\ell + \bar x_\ell$. 
\end{center}
By means of \cite[{Eqs.\ (3n) and (3n'), Section~3.3}]{Comtet}, we then have
\begin{align*}
 B_{n,k}(\bar x_\ell) &= \!\sum_{\kappa\le k,\, \nu\le n} \binom{n}{\nu} 
 B_{\nu,\kappa}(-\bar e_\ell)B_{n-\nu,k-\kappa}(\bar x) \\
 &= \sum_{\kappa\le k} \binom{n}{(\ell+1)\kappa}\frac{((\ell+1)\kappa)!}{\kappa!} 
 (-1)^\kappa B_{n-(\ell+1)\kappa,k-\kappa}(\bar x) \\
 &= \frac{n!}{k!}\sum_{\kappa=0}^{k-1} (-1)^\kappa \binom{k}{\kappa} \binom{n-(\ell+1)\kappa-1}{k-\kappa-1} 
 + (-1)^k\frac{n!}{k!(n-(\ell+1)k)!} B_{n-(\ell+1)\kappa,0}(\bar x) \\
 &= \frac{n!}{k!} \left(\sum_{\kappa=0}^{k-1} (-1)^\kappa \binom{k}{\kappa} 
 \binom{n-(\ell+1)\kappa-1}{k-\kappa-1}+ (-1)^k \delta_{n,(\ell+1)k} \right).
\end{align*}
\end{proof}

\begin{identity}\label{Bell_Identity4}
For $r\ge 0$ let $t_j(r) =\binom{j+r}{r+1}$. Then
\begin{equation*}
  B_{n, k}(1! t_1(r), 2! t_2(r), \cdots) = \frac{n!}{k!}\binom{n+(r+1)k-1}{n-k}.
\end{equation*}
\end{identity}
\begin{proof}
For $k=1$ and all $n$ we have 
\begin{equation*}
  B_{n, 1}(1! t_1(r), 2! t_2(r), \cdots) = \frac{n!}{1!}\binom{n+r}{r+1} = \frac{n!}{1!}\binom{n+r}{n-1}.
\end{equation*}
Now we proceed by induction on $k$. For $k>1$, 
\begin{align*}
 B_{n, k}(1! t_1(r), 2! t_2(r), \cdots)  
 &= \frac1{k} \sum_{j= 0}^n \binom{n}{j} j!\binom{j+r}{r+1}  B_{n-j, k-1}(1! t_1(r), 2! t_2(r), 3! t_3(r), \cdots)\\
 &= \frac1{k} \sum_{j= 1}^n \binom{n}{j} j!\binom{j+r}{r+1}  \frac{(n-j)!}{(k-1)!} \binom{n-j+(r+1)(k-1)-1}{n-j-k+1}\\
 &= \frac{n!}{k!}\sum_{j= 1}^{n-k+1} \binom{j+r}{j-1}  \binom{n-j+(r+1)(k-1)-1}{n-j-k+1} \\
 &= \frac{n!}{k!}\sum_{j= 0}^{n-k} \binom{j+r+1}{j}  \binom{n-j+(r+1)(k-1)-2}{n-k-j} \\
 &= \frac{n!}{k!}\sum_{j= 0}^{n-k} (-1)^{n-k} \binom{-r-2}{j} \binom{-(r+1)(k-1)-k+1}{n-k-j}.
\end{align*}
Using Chu-Vandermonde's identity, we arrive at
\begin{equation*}
 B_{n, k}(1! t_1(r), 2! t_2(r), \cdots) = (-1)^{n-k} \frac{n!}{k!}\binom{-(r+1)k-k}{n-k} 
 = \frac{n!}{k!}\binom{n + (r+1)k - 1}{n-k} .
\end{equation*}
\end{proof}


\begin{thebibliography}{9}

\bibitem{BBK08} H.~Belbachir, S.~Bouroubi, A.~Khelladi, Connection between ordinary multinomials, Fibonacci numbers, Bell polynomials and discrete uniform distribution, {\it Ann. Math. Inform.} \textbf{35} (2008), 21--30.

\bibitem{BernsteinSloane} M.~Bernstein and N. J. A.~Sloane, Some
canonical sequences of integers, {\it Linear Algebra Appl.}
\textbf{226/228} (1995), 57--72.

\bibitem{BGW12} D.~Birmajer, J.~Gil, and M.~Weiner, Some convolution
identities and an inverse relation involving partial Bell polynomials,
{\it Electron. J. Combin.} \textbf{19} (4) (2012), Paper \#34.

\bibitem{BGW15a} D.~Birmajer, J.~Gil, and M.~Weiner, Linear recurrence
sequences and their convolutions via Bell polynomials, {\it J. Integer
Seq.} \textbf{18} (2015), 
\href{https://cs.uwaterloo.ca/journals/JIS/VOL18/Gil/gil3.html}{Article 15.1.2}.

\bibitem{Cameron} P. J.~Cameron, Some sequences of integers, {\it Discrete Math.} \textbf{75} (1989), 89--102. 

\bibitem{Comtet} L.~Comtet, {\it Advanced Combinatorics: The Art of Finite and Infinite Expansions}, D. Reidel Publishing Co., 1974.

\bibitem{Janjic15} M.~Janji\'c, On linear recurrence equations arising
from compositions of positive integers, {\it J. Integer Seq.}
\textbf{18} (2015), 
\href{https://cs.uwaterloo.ca/journals/JIS/VOL18/Janjic/janjic63.html}{Article 15.4.7}.

\bibitem{Sloane} N. J. A.~Sloane, The On-Line Encyclopedia of Integer
Sequences, \url{http://oeis.org}.

\bibitem{WW09} W.~Wang and T.~Wang, General identities on Bell
polynomials, {\it Comput. Math. Appl.}, \textbf{58} (2009), 
104--118.

\end{thebibliography}
\end{document}